\documentclass[english,letterpaper,11pt,reqno]{amsart}
\usepackage[backend=bibtex,style=alphabetic,maxnames=6]{biblatex}
\usepackage{appendix}
\usepackage{amsbsy}
\usepackage{scalerel}
\usepackage{tcolorbox}
\usepackage{amsfonts}
\usepackage{amsmath}
\usepackage{amssymb}
\usepackage{amsthm}
\usepackage{graphicx}
\usepackage{ifthen}
\usepackage{textcomp}
\usepackage{enumitem, color, amssymb}
\usepackage{dsfont}
\usepackage{mathrsfs}
\usepackage{calrsfs}
\usepackage[bookmarksnumbered,colorlinks]{hyperref}
\hypersetup{colorlinks=true, linkcolor=blue, citecolor=blue}
\usepackage{cancel}
\usepackage{setspace}

\newcommand{\lra}{\longrightarrow}

\newcommand{\RR}{\mathbb{R}}

\newcommand{\vep}{\varepsilon}

\makeatletter
\newcommand*{\defeq}{\mathrel{\rlap{%
                     \raisebox{0.25ex}{$\m@th\cdot$}}%
                     \raisebox{-0.25ex}{$\m@th\cdot$}}%
                     =}
\makeatother

\makeatletter
\newcommand*\owedge{\mathpalette\@owedge\relax}
\newcommand*\@owedge[1]{%
  \mathbin{%
    \ooalign{%
      $#1\m@th\bigcirc$\cr
      \hidewidth$#1\m@th\wedge$\hidewidth\cr
    }%
  }%
}
\makeatother

\newtheorem{thm}{Theorem}
\newtheorem{lemma}{Lemma}
\newtheorem{cor}{Corollary}

\newtheorem{prop}{Proposition}
\newtheorem*{definition-non}{Definition}
\newtheorem*{theorem-non}{Theorem}
\newtheorem*{proposition-non}{Proposition}
\newtheorem*{lemma-non}{Lemma}
\newtheorem*{corollary-non}{Corollary}

\newcommand{\beqa}{\begin{eqnarray}}
\newcommand{\beq}{\begin{equation}}
\newcommand{\eeqa}{\end{eqnarray}}
\newcommand{\eeq}{\end{equation}}

\newcommand\ipl[2]{\langle {#1},{#2}\rangle_{\!g_{\scalebox{0.3}{\emph{L}}}}}
\newcommand\ipr[2]{\langle {#1},{#2}\rangle_{\!\scalebox{0.7}{\emph{g}}}}

\newcommand\ww[2]{#1 \wedge #2}
\newcommand\imp{\hspace{.2in}\Rightarrow\hspace{.2in}}

\newcommand\cd[2]{\nabla_{\!#1}{#2}}
\newcommand\cdL[2]{\nabla^{\scalebox{0.4}{$L$}}_{\!#1}{#2}}

\newcommand\gL{g_{\scalebox{0.4}{$L$}}}
\newcommand\TgL{\tilde{g}_{\scalebox{0.4}{$L$}}}

\newcommand\Ric{\text{Ric}}
\newcommand\RicL{\text{Ric}_{\scalebox{0.4}{$L$}}}
\newcommand\Rmr{\text{Rm}}
\newcommand\RmL{\text{Rm}_{\scalebox{0.4}{$L$}}}

\newcommand\LCL{\nabla^{\scalebox{0.4}{$L$}}}
\newcommand\comma{\hspace{.2in},\hspace{.2in}}
\newcommand\commas{\hspace{.1in},\hspace{.1in}}

\newcommand\co{\hat{R}}

\newtheorem{innercustomgeneric}{\customgenericname} 
\providecommand{\customgenericname}{}
\newcommand{\newcustomtheorem}[2]{%
  \newenvironment{#1}[1]
  {%
   \renewcommand\customgenericname{#2}%
   \renewcommand\theinnercustomgeneric{##1}%
   \innercustomgeneric
  }
  {\endinnercustomgeneric}
}

\newcustomtheorem{customthm}{Theorem}
\newcustomtheorem{customlemma}{Lemma}

\begin{document}
\title[]{On the Curvature and Topology of Compact Stationary Spacetimes}
\author[]{Amir Babak Aazami}
\address{Clark University\hfill\break\indent
Worcester, MA 01610}
\email{aaazami@clarku.edu}

\maketitle
\begin{abstract}
Using the result of Petersen \& Wink '21, we find obstructions to the curvature and topology of compact Lorentzian manifolds admitting a unit-length timelike Killing vector field.
\end{abstract}

\section{Introduction}
In this article we apply the recent work \cite{PW} to find curvature and topological obstructions to compact \emph{stationary spacetimes}; i.e., compact Lorentzian manifolds with a timelike Killing vector field.  Little is known about the Betti numbers of such manifolds other than the following: \cite{kamishima} showed that compact flat stationary Lorentzian manifolds must have nonzero first Betti number; \cite{RS96} then improved upon this result in two respects: By using Bochner's technique, they showed that, up to covering, only tori satisfy this, and that compact Ricci-flat stationary Lorentzian manifolds must also have nonzero first Betti number. (For other obstructions, e.g., regarding sectional curvature, or isometry groups, or for Kundt spacetimes in place of stationary ones, see, e.g., \cite{CM,markvorsen,carriere,Klingler,baum,piccione,schliebner,RS2}.) Here we show that, by using \cite{PW}, further information can be gained about higher Betti numbers. In the latter, the Lichnerowicz Laplacian\,---\,which extends the Bochner method by giving conditions for when harmonic tensors are parallel\,---\,was used to relate Betti numbers $b_i$ to the eigenvalues of the curvature operator of a closed Riemannian $n$-manifold $(M,g)$.  In particular, by generalizing a result in \cite{poor}, \cite{PW} showed that, for $n \geq 3$ and $1 \leq p \leq \lfloor \frac{n}{2}\rfloor$, if $g$'s curvature operator is $(n-p)$-positive\,---\,i.e., the sum of its smallest $n-p$ eigenvalues is positive\,---\,then $b_1=\cdots = b_p = 0$ and $b_{n-p} = \cdots = b_{n-1} = 0$. We apply this result here to compact Lorentzian $n$-manifolds $(M,\gL)$ admitting a timelike Killing vector field $T$, by analyzing the Riemannian metric
\beqa
\label{eqn:gR}
g \defeq \gL - 2\frac{T^\flat \otimes T^\flat}{\gL(T,T)} \comma T^\flat \defeq \gL(T,\cdot).
\eeqa
Not only are compact stationary spacetimes an appealing class of manifolds to study\,---\,e.g., they are always geodesically complete \cite{RS95}, something that is not guaranteed in the Lorentzian setting\,---\,but every stationary spacetime admits a ``normalized" one in its conformal class (i.e., $\gL(T,T)=-1$ with $T$ still a timelike Killing vector field). As we are interested only in the topology of $M$, restricting our attention to such normalized stationary spacetimes is not an impediment\,---\,especially given the complicated relationship between the curvature tensors of $g$ and $\gL$. However, as we show in Proposition \ref{prop:basis}, the virtue of a unit-length Killing vector field $T$ is that there always exist canonical local frames with respect to which the curvature operators of $g$ and $\gL$ are so closely aligned that the former effectively becomes the ``symmetrization" of the latter. (This overcomes the fact that the curvature operator of a Lorentzian metric is neither symmetric nor diagonalizable in general, and can have complex eigenvalues.) As a consequence, we then use \cite{PW} to prove the following general results:
\begin{theorem-non}
Let $(M,\gL)$ be a closed, connected, and oriented Lorentzian 3-manifold with a unit-length timelike Killing vector field $T$. Relative to the local $\gL$-orthonormal frame \eqref{eqn:even}, set \emph{$\RmL(T,X_i,X_j,X_k) \defeq R_{Tijk}$}, etc., where \emph{$\RmL$} is the Riemann curvature 4-tensor of $\gL$.  If the symmetric matrix
$$
\begin{bmatrix}
-R_{T1T1} & 0 & R_{12T1}\\
0 & -R_{T2T2} & R_{12T2}\\
R_{T112} & R_{T212} & -R_{1212}+3R_{T1T1}+3R_{T2T2}
\end{bmatrix}
$$
is 2-positive, then $b_1 = b_2 = 0$. In general, if $(M,\gL)$ is odd-dimensional and the ${n\choose 2}\times {n \choose 2}$ analogue of this matrix is $(n-p)$-positive for $1 \leq p \leq \lfloor \frac{n}{2}\rfloor$, then $b_1=\cdots=b_p=b_{n-p}=\cdots=b_{n-1}=0$.
\end{theorem-non}

(The component $R_{T1T2} = 0$ in the frame \eqref{eqn:even}.) As an application, the $n$-torus\,---\,usually a source of examples of distinguished Lorentzian metrics \cite{RS96,S97}\,---\,can never admit such $(n-p)$-positive matrices (nor of those in the Theorem below). Another application is:

\begin{corollary-non}
Let $(M,\gL)$ be a closed, connected, and oriented Lorentzian 3-manifold with a unit-length timelike Killing vector field $T$, Ricci tensor \emph{$\RicL$}, and scalar curvature \emph{$\text{scal}_{\scalebox{0.4}{$L$}}$}. If $T$ is not parallel, and if
\emph{
$$
\text{scal}_{\scalebox{0.4}{$L$}}> 3\RicL(T,T) + |d\,\text{ln}(\RicL(T,T))|_{g_{\scalebox{0.3}{$L$}}}^2,
$$}then $b_1=b_2=0$. If $\gL$ is replaced by a Riemannian metric $g$, then the inequality becomes \emph{$\text{scal} > \Ric(T,T) + |d\,\text{ln}(\Ric(T,T))|_{g}^2$}.
\end{corollary-non}

The even-dimensional case is similar to the odd-dimensional case above, but here we also get obstructions to $k$-positivity of the eigenvalues:

\begin{theorem-non}
If $(M,\gL)$ is four-dimensional, then relative to the local $\gL$-orthonormal frame \eqref{eqn:even}, the symmetric matrix
$$
\begin{bmatrix}
-R_{T1T1} & R_{T2T1} & R_{T3T1} & R_{23T1} & R_{31T1} & R_{12T1}\\
R_{T1T2} & -R_{T2T2} & R_{T3T2} & R_{23T2} & R_{31T2} & R_{12T2}\\
R_{T1T3} & T_{T2T3} & -R_{T3T3} & R_{23T3} & R_{31T3} & R_{12T3}\\
R_{T123} & R_{T223} & R_{T323} & -R_{2323}+3R_{T2T2}+3R_{T3T3} & -R_{3123} & -R_{1223}\\
R_{T131} & R_{T231} & R_{T331} & -R_{2331} & -R_{3131} & -R_{1231}\\
R_{T112} & R_{T212} & R_{T312} & -R_{2312} & -R_{3112} & -R_{1212}
\end{bmatrix}
$$
cannot be 3-positive. If $(M,\gL)$ is $2n$-dimensional, then the ${2n\choose 2}\times {2n \choose 2}$ analogue of this matrix cannot be $n$-positive\emph{;} if it is $(2n-p)$-positive for $1 \leq p < n$, then $b_1=\cdots=b_p=b_{2n-p}=\cdots=b_{2n-1}=0$.
\end{theorem-non}

(In fact $R_{T1Ti} = R_{T2T3} = 0$ in the frame \eqref{eqn:even}.) The eigenvalues in both theorems are frame-independent, because these matrices are in fact the curvature operators, not of the Lorentzian metric $\gL$, but of the Riemannian metric \eqref{eqn:gR} expressed in Lorentzian data. The reason they take on such simple forms in terms of $\RmL$ is because of the existence of the distinguished frames mentioned above, particularly in dimensions three and four. 

\section{Relationship between $g$ and $\gL$}
The setup of our problem is as follows: Let $(M,\gL)$ be a closed, oriented Lorentzian $n$-manifold endowed with a timelike Killing vector field $T$; i.e., a compact and oriented \emph{stationary spacetime} without boundary. Note that $T$ will necessarily be a unit-length timelike Killing vector field of the conformal metric
\beqa
\label{eqn:conf}
\tilde{g}_{\scalebox{0.4}{$L$}} \defeq \frac{\gL}{-\gL(T,T)}\cdot
\eeqa
Now form the Riemannian metric
\beqa
\label{eqn:Riem}
g \defeq \gL - 2\frac{T^\flat \otimes T^\flat}{\gL(T,T)} \comma T^\flat \defeq \gL(T,\cdot).
\eeqa
Note that $g(T,T) = -\gL(T,T) > 0$, that every $\gL$-orthogonal local frame of the form $\{T,X_1,\dots,X_{n-1}\}$ is also $g$-orthogonal (in fact $g(T,\cdot) = -\gL(T,\cdot)$ and $g(X_i,\cdot) = \gL(X_i,\cdot)$), and finally that $T$ is also a Killing vector field with respect to $g$: $\mathfrak{L}_Tg = 0$ if and only if $\mathfrak{L}_Tg_{\scalebox{0.4}{$L$}} = 0$, where $\mathfrak{L}$ is the Lie derivative. Thus if $M$ admits a Riemannian metric $g$ with a nowhere vanishing Killing vector field $T$, then $M$ admits a stationary Lorentzian metric $\gL$, and thus also a normalized one via \eqref{eqn:conf}.  As shown, e.g.,  in \cite[Prop.~5]{RS2}, all of this can be accomplished by finding an $\mathbb{S}^1$-action on $M$ without fixed points. (For more general Killing vector fields, see \cite{flores2}.) With these preliminaries established, we would now like to relate the Riemann curvature 4-tensors of $g$ and $\gL$; in order to do so, let us first relate their Levi-Civita connections:

\begin{prop}
\label{prop:LC}
Let $(M,\gL)$ be a Lorentzian $n$-manifold with a timelike Killing vector field $T$. Let $g$ be the Riemannian metric defined via \eqref{eqn:Riem}. With respect to any $\gL$- and $g$-orthogonal local frame $\{T,X_1,\dots,X_{n-1}\}$ with each $X_i$ having unit length, the Levi-Civita connections $\nabla$ and $\LCL$ are related as follows:
\beqa
\arraycolsep=1.4pt\def\arraystretch{1.5}
\left\{\begin{array}{lr}
\cd{T}{T} = -\cdL{T}{T},\\
\cd{X_i}{T} = -\cdL{X_i}{T} + X_i(\emph{\text{ln}}(g(T,T))T,\\
\cd{X_i}{X_j} = \cdL{X_i}{X_j},\\
\cd{T}{X_i} = \cdL{T}{X_i} -2\cdL{X_i}{T} + X_i(\emph{\text{ln}}(g(T,T))T. \label{eqn:LCs}
\end{array}\right.
\eeqa
\end{prop}

\begin{proof}
A Koszul formula for $\nabla-\LCL$ is derived in \cite[Proposition~2.3]{olea}, in the case of a general unit-length timelike vector field $T$. Modifying that formula to our setting\,---\,with $T$ a Killing vector field, though not necessarily unit length\,---\,the relations \eqref{eqn:LCs} can be straightforwardly derived. Indeed, 
\beqa
g(\cd{X_i}{X_j},X_k) \!\!&=&\!\! \frac{1}{2}\Big[X_i(g(X_j,X_k)) + X_j(g(X_k,X_i)) - X_k(g(X_i,X_j))\nonumber\\
&&\hspace{.2in} -g(X_j,[X_i,X_k]) - g(X_k,[X_j,X_i]) + g(X_i,[X_k,X_j])\Big]\nonumber
\eeqa
is term-by-term equivalent to its counterpart $\gL(\cdL{X_i}{X_j},X_k)$, as $g(X_\alpha,\cdot) = \gL(X_\alpha,\cdot)$ for $\alpha=i,j,k$. Thus
$$
\gL(\cdL{X_i}{X_j},X_k) = g(\cd{X_i}{X_j},X_k) \overset{\eqref{eqn:Riem}}{=} \gL(\cd{X_i}{X_j},X_k)-2\frac{\gL(T,\cd{X_i}{X_j})\cancelto{0}{\gL(T,X_k)}}{\gL(T,T)},
$$
so that
\beqa
\label{eqn:first}
\gL(\cdL{X_i}{X_j}-\cd{X_i}{X_j},X_k) = 0.
\eeqa
Now putting $T$ in place of $X_k$,
\beqa
g(\cd{X_i}{X_j},T) \!\!&=&\!\! \frac{1}{2}\Big[X_i(g(X_j,T)) + X_j(g(T,X_i)) - T(g(X_i,X_j))\nonumber\\
&&\hspace{.2in} -g(X_j,[X_i,T]) - g(T,[X_j,X_i]) + g(X_i,[T,X_j])\Big]\nonumber
\eeqa
is term-by-term equivalent to its counterpart $\gL(\cdL{X_i}{X_j},T)$\,---\,except for the term
$$
g(T,[X_j,X_i]) \overset{\eqref{eqn:Riem}}{=} -\gL(T,[X_j,X_i]),
$$
so that
$$
\underbrace{\,\gL(\cdL{X_i}{X_j},T) + \gL(T,[X_j,X_i])\,}_{\text{$= \gL(\cdL{X_j}{X_i},T) \overset{(*)}{=} -\gL(\cdL{X_i}{X_j},T)$}} = g(\cd{X_i}{X_j},T) \overset{\eqref{eqn:Riem}}{=} -\gL(\cd{X_i}{X_j},T)
$$
where $(*)$ holds because $T$ is a Killing vector field,
$$
\mathfrak{L}_{T}g_{\scalebox{0.3}{\emph{L}}} = 0 \imp \gL(\cdL{X_i}{T},X_j) = -\gL(\cdL{X_j}{T},X_i).
$$
Thus $\gL(\cdL{X_i}{X_j}-\cd{X_i}{X_j},T) = 0$.  Combining this with \eqref{eqn:first} yields
$$
\gL(\cdL{X_i}{X_j}-\cd{X_i}{X_j},\cdot) = 0,
$$
hence $\cdL{X_i}{X_j} = \cd{X_i}{X_j}$. Using this, $\cd{X_i}{T} = -\cdL{X_i}{T} + X_i(\text{ln}(g(T,T))T$ can be derived, from which $\cd{T}{X_i} = \cdL{T}{X_i} -2\cdL{X_i}{T} + X_i(\text{ln}(g(T,T))T$ follows immediately via $\cd{T}{X_i}-\cd{X_i}{T} = [T,X_i] = \cdL{T}{X_i}-\cdL{X_i}{T}$. Finally, using the Killing condition, the case $\cd{T}{T} = -\cdL{T}{T}$ follows easily, via
$$
g(\cd{T}{T},X_i) = -\frac{1}{2}X_i(g(T,T)) = \gL(\cdL{X_i}{T},T) = -\gL(\cdL{T}{T},X_i) = -g(\cdL{T}{T},X_i)
$$
and the fact that $g(\cd{T}{T},T) = g(-\cdL{T}{T},T) = 0$.
\end{proof}

Armed with Proposition \ref{prop:LC}, we now relate the Riemann curvature 4-tensors of $g$ and $\gL$. In what follows, our sign convention for the Riemann curvature endomorphism is
$
R_{\scalebox{0.4}{$L$}}(X,Y)Z = \cd{X}{\cd{Y}{Z}} - \cd{Y}{\cd{X}{Z}} - \cd{[X,Y]}{Z}.
$
Although this sign is the opposite of that in \cite{PW}, this is compensated for by the fact that our curvature operator \eqref{eqn:co0} below is defined with a minus sign, which \cite{PW} does not have.

\begin{prop}
\label{prop:R}
With respect to the local frame $\{T,X_1,\dots,X_{n-1}\}$ of Proposition \ref{prop:LC}, the components of the Riemann curvature 4-tensors \emph{$\Rmr$} and \emph{$\RmL$} of $g$ and $\gL$, respectively, are related as follows:
\end{prop}
\vskip-16pt
\beqa
\arraycolsep=1.4pt\def\arraystretch{1.5}
\left\{\begin{array}{lr}
\Rmr(X_i,X_j,T,X_k) = -\RmL(X_i,X_j,T,X_k),\\
\Rmr(T,X_i,T,X_j) = -\RmL(T,X_i,T,X_j) - 2\gL(\cdL{X_i}{T},\cdL{X_j}{T})\\
\hspace{2.2in}+\,\frac{1}{2g_{\scalebox{0.3}{$L$}}(T,T)}X_i(g_{\scalebox{0.3}{$L$}}(T,T))X_j(g_{\scalebox{0.3}{$L$}}(T,T))\label{eqn:RmL}\\
\Rmr(X_i,X_j,X_k,X_l) = \RmL(X_i,X_j,X_k,X_l)\\
\hspace{1.3in} +\,\frac{2}{g_{\scalebox{0.3}{$L$}}(T,T)}\Big[\gL(\cdL{X_i}{T},X_l)\gL(\cdL{X_j}{T},X_k)\\
\hspace{2.1in}-\,\gL(\cdL{X_i}{T},X_k)\gL(\cdL{X_j}{T},X_l)\\
\hspace{2.5in} -2\gL(\cdL{X_i}{T},X_j)\gL(\cdL{X_k}{T},X_l)\Big].
\label{eqn:Rm}\\
\end{array}\right.
\eeqa

\begin{proof}
Most of these components can be found in \cite[Section~3]{olea} in the case of a general unit-length timelike vector field $T$. As we will need all components of $\Rmr,\RmL$, let us derive in detail the first two components in \eqref{eqn:Rm} here (and also given that our $T$ may not have unit length). Thus, consider the component $\Rmr(X_i,X_j,T,X_k)$, and begin by writing it as
\beqa
\Rmr(X_i,X_j,T,X_k) \!\!&=&\!\! g(\cd{X_i}{\!\cd{X_j}{T}}-\cd{X_j}{\!\cd{X_i}{T}}-\cd{[X_i,X_j]}{T},X_k)\nonumber\\
&=&\!\! X_i(g(\cd{X_j}{T},X_k)) - g(\cd{X_j}{T},\cd{X_i}{X_k})\nonumber\\
&& -X_j(g(\cd{X_i}{T},X_k)) + g(\cd{X_i}{T},\cd{X_j}{X_k})\nonumber\\
&&\hspace{1.5in}-g(\cd{[X_i,X_j]}{T},X_k).\nonumber
\eeqa
Using Proposition \ref{prop:LC}, the Killing condition, and the fact that $g(\cdot,X_k) = \gL(\cdot,X_k)$, the terms on the right-hand side can be written as
\beqa
\Rmr(X_i,X_j,T,X_k) \!\!&=&\!\! -X_i(\gL(\cdL{X_j}{T},X_k)) + \gL(\cd{\cd{X_i}{X_k}}{T},X_j)\nonumber\\
&&\ \ +\, X_j(\gL(\cdL{X_i}{T},X_k)) - \gL(\cd{\cd{X_j}{X_k}}{T},X_i)\label{eqn:Rstep1}\\
&&\hspace{1.6in}-\gL(\cd{[X_i,X_j]}{T},X_k).\nonumber
\eeqa
Next, expanding $\cd{X_i}{X_k} = cT+ b^\lambda X_\lambda = \cdL{X_i}{X_k}$, we have that
\beqa
\gL(\cd{\cd{X_i}{X_k}}{T},X_j) \!\!&=&\!\! c\gL(\cd{T}{T},X_j)+b^\lambda \gL(\cd{X_\lambda}{T},X_j)\nonumber\\
&\overset{\eqref{eqn:LCs}}{=}&\!\! -c\gL(\cdL{T}{T},X_j)-b^\lambda \gL(\cdL{X_\lambda}{T},X_j)\nonumber\\
&=&\!\! -\gL(\cdL{\cdL{X_i}{X_k}}{T},X_j).\nonumber
\eeqa
Likewise, $\gL(\cd{\cd{X_j}{X_k}}{T},X_i) = -\gL(\cdL{\cdL{X_j}{X_k}}{T},X_i)$.  A similar expansion of $[X_i,X_j]$ yields $\gL(\cd{[X_i,X_j]}{T},X_k) = -\gL(\cdL{[X_i,X_j]}{T},X_k)$. Inserting these into \eqref{eqn:Rstep1} and simplifying yields
$$
\Rmr(X_i,X_j,T,X_k) = -\RmL(X_i,X_j,T,X_k).
$$
A similar computation occurs for the component
\beqa
\Rmr(X_i,T,T,X_j) \!\!&=&\!\! g(\cd{X_i}{\!\cd{T}{T}}-\cd{T}{\!\cd{X_i}{T}}-\cd{[X_i,T]}{T},X_j)\nonumber\\
&=&\!\! X_i(g(\cd{T}{T},X_j)) - g(\cd{T}{T},\cd{X_i}{X_j})\nonumber\\
&& -T(g(\cd{X_i}{T},X_j)) + g(\cd{X_i}{T},\cd{T}{X_j})\label{eqn:Rstep3}\\
&&\hspace{1.5in}-g(\cd{[X_i,T]}{T},X_j).\nonumber
\eeqa
Once again, the terms here transform identically into (minus) their $\gL$- and $\nabla^{\scalebox{0.4}{$L$}}$-counterparts\,---\,except for $g(\cd{X_i}{T},\cd{T}{X_j})$, which, using Proposition \ref{prop:LC} and the Killing condition, yields the additional terms
\beqa
g(\cd{X_i}{T},\cd{T}{X_j}) \!\!&=&\!\! -\gL(\cdL{X_i}{T},\cdL{T}{X_j}) + 2\gL(\cdL{X_i}{T},\cdL{X_j}{T})\nonumber\\
&&\hspace{1in} - \frac{X_i(g_{\scalebox{0.4}{$L$}}(T,T))X_j(g_{\scalebox{0.4}{$L$}}(T,T))}{2g_{\scalebox{0.4}{$L$}}(T,T)}\cdot\nonumber
\eeqa
Inserting these back into \eqref{eqn:Rstep3}, we have
\beqa
\Rmr(X_i,T,T,X_j) \!\!&=&\!\! -\RmL(X_i,T,T,X_j) + 2\gL(\cdL{X_i}{T},\cdL{X_j}{T})\nonumber\\
&&\hspace{1in}-\frac{X_i(g_{\scalebox{0.4}{$L$}}(T,T))X_j(g_{\scalebox{0.4}{$L$}}(T,T))}{2g_{\scalebox{0.4}{$L$}}(T,T)},\nonumber
\eeqa
which implies \eqref{eqn:RmL}. We omit the details of the computation for the component $\Rmr(X_i,X_j,X_k,X_l)$, as they are similar to the cases above.
\end{proof}

\section{Motivating Example}

Now, for a given compact stationary spacetime $(M,\gL)$, one can, of course, draw topological information simply by applying \cite{PW} to the curvature operator of the Riemannian metric
\beqa
\label{eqn:1*}
g \defeq \gL - 2\frac{T^\flat \otimes T^\flat}{\gL(T,T)} \comma T^\flat \defeq \gL(T,\cdot).
\eeqa
Indeed, if the curvature operator of $g$ is $(n-p)$-positive for $1 \leq p \leq \lfloor \frac{n}{2}\rfloor$, then \cite{PW} directly yields $b_1=\cdots = b_p = 0$ and $b_{n-p} = \cdots = b_{n-1} = 0$.  (Recall that the \emph{curvature operator} $\co\colon \Lambda^2\lra \Lambda^2$ of $g$ is the linear endomorphism defined by  
\beqa
\label{eqn:co0}
\ipr{\co(v\wedge w)}{x\wedge y} \defeq -\Rmr(v,w,x,y)\hspace{.2in}\text{for all $v,w,x,y \in T_pM$},
\eeqa
where $\ipr{\,}{}$ is the $g$-induced inner product $\ipr{\,}{}$ on $\Lambda^2$:
$$
\ipr{\ww{v}{w}}{\ww{x}{y}}  \defeq \text{det}\begin{bmatrix}
g(v,x) & g(v,y)\\
g(w,x) & g(w,y)
\end{bmatrix}\cdot
$$
The curvature operator $\hat{R}_{\scalebox{0.4}{$L$}}$ of $\gL$ is defined in the same way, but with $\RmL$ in place of $\Rmr$ and $\ipl{\,}{}$ in place of $\ipr{\,}{}$.) However, our goal here is to use Proposition \ref{prop:R} to work with the curvature operator of $\gL$ \emph{directly}. Indeed, we would like to view \cite{PW} as allowing us to bypass the fact that the machinery of the Lichnerowicz Laplacian has no Lorentzian analogue, or indeed the more basic fact that the curvature operator of a Lorentzian metric is not symmetric in general (see \eqref{eqn:sym0} and \eqref{eqn:sym1}), hence may not have real eigenvalues to begin with.  To that end, glancing at the non-$\RmL$-terms in the right-hand side of \eqref{eqn:RmL}, a natural question arises: In what cases will these terms, or at least most of them, vanish?  Because if they do, then the Riemannian curvature operator $\co$ of $g$ can be regarded as the ``symmetrization" of the curvature operator $\hat{R}_{\scalebox{0.4}{$L$}}$ of $\gL$\,---\,and thus topological information can be drawn, via \cite{PW}, ``purely from Lorentzian data." To illustrate this point, consider the following example: The manifold $\mathbb{S}^3$ with the following Lorentzian metric in local coordinates $(t,\theta_1,\theta_2)$, with $t \in \big[0,\frac{\pi}{2}\big]$:
\beqa
\label{cor:check}
\gL = \begin{pmatrix}
1 & 0 & 0\\
0 & \sin^2\!t(1-2\sin^2\!t) & -2\sin^2\!t\cos^2\!t\\
0 & -2\sin^2\!t\cos^2\!t & \cos^2\!t(1-2\cos^2\!t)
\end{pmatrix}\cdot
\eeqa
The vector field $T \defeq \partial_{\theta_1}+\partial_{\theta_2}$ is a unit-length timelike Killing vector field with respect to $\gL$. But it is more than that, for in fact the corresponding Riemannian metric \eqref{eqn:1*} is the standard (round) metric on $\mathbb{S}^3$ and $T$ is the well known Hopf Killing vector field:
\beqa
\label{eqn:round}
g = \gL - 2\frac{T^\flat \otimes T^\flat}{\gL(T,T)} = dt^2 + \sin^2\!t\,d\theta_1^2 + \cos^2\!t\,d\theta_2^2.
\eeqa
(See, e.g., (see \cite[p.~23.]{PP}.) With respect to the $g$- and $\gL$-orthonormal basis
\beqa
\label{eqn:simple0}
T = \partial_{\theta_1}+\partial_{\theta_2} \comma X_1 \defeq \partial_t \comma X_2 \defeq \cot\!t\,\partial_{\theta_1}-\tan\!t\,\partial_{\theta_2},
\eeqa
it is straightforward to verify that 
\beqa
\label{eqn:simple}
\arraycolsep=1.4pt\def\arraystretch{1.5}
\left\{\begin{array}{lr}
\cd{T}{T} = 0 = \cdL{T}{T},\\
\cd{X_1}{T} = X_2 = -\cdL{X_1}{T},\\
\cd{X_2}{T} = -X_1 = -\cdL{X_2}{T},
\end{array}\right.
\eeqa
that $[T,X_1] = [T,X_2] = 0, [X_1,X_2] = -2T-(\cot t-\tan t)X_2$, and that
\beqa
\label{eqn:simple2}
\arraycolsep=1.4pt\def\arraystretch{1.5}
\left\{\begin{array}{lr}
\cd{X_1}{X_1} = 0 = \cdL{X_1}{X_1},\\
\cd{X_1}{X_2} = -T = \cdL{X_1}{X_2},\\
\cd{X_2}{X_1} = T+(\cot t-\tan t)X_2 = \cdL{X_2}{X_1}.
\end{array}\right.
\eeqa
If we now express the curvature operator $\co$ of $g$ with respect to the $\ipr{\,}{}$-orthonormal basis
$
\{\ww{T}{X_1}\,,\,\ww{T}{X_2}\,,\,\ww{X_1}{X_2}\} \subseteq \Lambda^2
$ (recall that $\Lambda^2$ is ${n \choose 2}$-dimensional in general), it will be given by
$$
\co = -\!\begin{bmatrix}
\Rmr(T,X_1,T,X_1) & \Rmr(T,X_2,T,X_1) & \Rmr(X_1,X_2,T,X_1)\\
\Rmr(T,X_1,T,X_2) & \Rmr(T,X_2,T,X_2) & \Rmr(X_1,X_2,T,X_2)\\
\Rmr(T,X_1,X_1,X_2) & \Rmr(T,X_2,X_1,X_2) & \Rmr(X_1,X_2,X_1,X_2)
\end{bmatrix}\cdot
$$
But, using \eqref{eqn:simple}, \eqref{eqn:simple2}, and \eqref{eqn:RmL}, we can rewrite this entirely in terms of $\RmL$,
\beqa
\label{eqn:sym00}
\begin{bmatrix}
\RmL(T,X_1,T,X_1)+2 & \RmL(T,X_2,T,X_1) & \RmL(X_1,X_2,T,X_1)\\
\RmL(T,X_1,T,X_2) & \RmL(T,X_2,T,X_2)+2 & \RmL(X_1,X_2,T,X_2)\\
\RmL(T,X_1,X_1,X_2) & \RmL(T,X_2,X_1,X_2) & -\RmL(X_1,X_2,X_1,X_2)-6
\end{bmatrix},
\eeqa
which simplifies to
$$
\co = \begin{bmatrix} -1+2&0&0\\ 0 &-1+2&0\\ 0&0&7-6\end{bmatrix}\cdot
$$
For $1 \leq p \leq \lfloor \frac{3}{2}\rfloor$, \cite{PW} now yields that because $\co$ is $(3-1)=2$-positive, the Betti numbers of $M=\mathbb{S}^3$ satisfy $b_1=b_{2} = 0$, which is, of course, true.  The salient point, however, is that because the frame \eqref{eqn:simple0} satisfied the relations \eqref{eqn:simple}, the Lorentzian curvature operator $\hat{R}_{\scalebox{0.4}{$L$}}$ of $\gL$,
\beqa
\label{eqn:sym0}
\hat{R}_{\scalebox{0.4}{$L$}} = \begin{bmatrix}
\RmL(T,X_1,T,X_1) & \RmL(T,X_2,T,X_1) & \RmL(X_1,X_2,T,X_1)\\
\RmL(T,X_1,T,X_2) & \RmL(T,X_2,T,X_2) & \RmL(X_1,X_2,T,X_2)\\
-\RmL(T,X_1,X_1,X_2) & -\RmL(T,X_2,X_1,X_2) & -\RmL(X_1,X_2,X_1,X_2)\end{bmatrix}
\eeqa
is very closely aligned with $\co$. Indeed, the diagonal entries $2,2,-6$ nothwithstanding, observe how \eqref{eqn:sym00} is precisely the ``symmetrization" of $\hat{R}_{\scalebox{0.4}{$L$}}$ that we mentioned in the Introduction.

\section{Distinguished Local Frames}
As we now show, this is not unique to the frame \eqref{eqn:simple0} or to the manifold $\mathbb{S}^3$; on the contrary, whenever $T$ has unit length, such frames will always lie at our disposal (cf.~\cite{RS96} and \cite[Section~5.3(c)]{S97}). In Section \ref{sec:main}, we will show that with respect to these frames the curvature operators $\co$ and $\hat{R}_{\scalebox{0.4}{$L$}}$ will be so closely aligned that $\co$ will effectively ``symmetrize" $\hat{R}_{\scalebox{0.4}{$L$}}$.

\begin{prop}
\label{prop:basis}
Let $(M,\gL)$ be a Lorentzian $n$-manifold with $n\geq 3$ and with unit-length timelike Killing vector field $T$\emph{;} assume $T$ is not parallel. At every point of $M$, there is a local orthonormal frame $\{T,X_1,\dots,X_{n-1}\}$ such that an even number of the $X_i$'s satisfies
\beqa
\label{eqn:even}
\cdL{X_i}{T} = f_{ii+1} X_{i+1} \comma \cdL{X_{i+1}}{T} = -f_{ii+1} X_i,
\eeqa
for some functions $f_{ii+1} < 0$, with any remaining $X_j$'s, along with $T$, satisfying $\cdL{T}{T} = \cdL{X_j}{T} = 0$. There are $\leq \lfloor\frac{n-1}{2}\rfloor$ such functions $f_{ii+1}$, with each $T(f_{ii+1}) = 0$ and each
\emph{$$
\RmL(X_{i},T,T,X_{i}) = f_{ii+1}^2 = \RmL(X_{i+1},T,T,X_{i+1}).
$$}Finally, \emph{$\RmL(X_i,T,T,X_{j}) = 0$} if $i\neq j$.
\end{prop}

\begin{proof}
(E.g., for the frame \eqref{eqn:simple0}, $f_{12} = -1$ and only $\cdL{T}{T} = 0$.) Observe that for $n=3$, \emph{any} orthonormal frame $\{T,X_1,X_2\}$ satisfies $\cdL{T}{T}=0$ and
$$
\cdL{X_1}{T} = f_{12}X_2\comma \cdL{X_2}{T}=-f_{12}X_1
$$
for some function $f_{12}$, simply by virtue of the Killing condition. But for $n\geq 4$ the Killing condition by itself yields only $\cdL{X_i}{T} = \sum_{j\neq i}\vep_{ij}f_{ij}X_j$, where $\vep_{ij} \defeq +1$ if $i < j$ and $-1$ if $i > j$, and $f_{ij} = f_{ji}$. Therefore, to prove the statement of the theorem for all $n\geq 3$, we proceed as follows. Consider the linear endomorphism $v \mapsto \cdL{v}{T}$ on $TM$ as well as its Riemannian counterpart, $v \mapsto \cd{v}{T}$, with $g$ given by \eqref{eqn:1*} as usual (with $\gL(T,T)=-1$). This endomorphism is skew-symmetric, but its composition with itself, namely, the linear endomorphism $\nabla T \circ \nabla T$ sending $v \mapsto \cd{\cd{v}{T}}{T}$, is $g$-self-adjoint, via two applications of the Killing condition:
$$
g(\cd{\cd{v}{T}}{T},w) = -g(\cd{w}{T},\cd{v}{T}) = g(v,\cd{\cd{w}{T}}{T}).
$$
As $g$ is positive-definite, $\nabla T \circ \nabla T$ thus has a local $g$-orthonormal basis of eigenvectors about every point of $M$.  Since the Killing vector field $T$ is assumed to have unit length, it follows that $\cd{T}{T} = 0$, so that $T$ can be taken as one of these basis eigenvectors; as a consequence, any basis including $T$ will necessarily be both $g$- and $\gL$-orthonormal.  We will now further modify this basis to produce the frame \eqref{eqn:even}. First, observe that the eigenvalues of $\nabla T \circ \nabla T$ must be nonpositive: If $v$ is a $g$-unit-length eigenvector with eigenvalue $\lambda$, then
\beqa
\label{eqn:vw}
\cd{\cd{v}{T}}{T} = \lambda v \imp \lambda = g(\cd{\cd{v}{T}}{T},v) = -g(\cd{v}{T},\cd{v}{T}) \leq 0,
\eeqa
with $\lambda = 0 \iff \cd{v}{T} = 0$. Furthermore, $\cdL{v}{T} = -\cd{v}{T}$ by Proposition \ref{prop:LC}, so that $\LCL T \circ \LCL T$ will have the same eigenvectors and eigenvalues. (E.g., the basis \eqref{eqn:simple0} had eigenvalues $0,-1,-1$.) Take now a $g$-unit-length eigenvector $v$ with nonzero eigenvalue $\lambda < 0$ (such a $\lambda$ must exist, as we are assuming that $T$ is not parallel), and set
$$
w \defeq \cd{v/{\sqrt{-\lambda}}}{T}.
$$
By \eqref{eqn:vw}, $w$ has unit length; but more than that, it is also an eigenvector of $\nabla T \circ \nabla T$ with eigenvalue $\lambda$,
$$
\cd{\cd{w}{T}}{T} = \cd{\cd{\cd{v/\sqrt{-\lambda}}{T}}{T}}{T} = \lambda \cd{v/\sqrt{-\lambda}}{T} = \lambda w,
$$
that is $g$-orthogonal to both $v$ and $T$:
$$
g(v,w) = g(v,\cd{v/\sqrt{-\lambda}}{T}) = 0 \commas g(T,w) = \frac{1}{2\sqrt{-\lambda}}v(g(T,T)) = 0.
$$
Another iteration brings us back to $v$ but with a sign change, precisely as we saw in \eqref{eqn:simple}:
$$
\cd{w}{T} = \cd{\cd{v/{\sqrt{-\lambda}}}{T}}{T} = \frac{\lambda}{\sqrt{-\lambda}}v=-\sqrt{-\lambda}\,v \comma \cd{v}{T} = \sqrt{-\lambda}\,w.
$$
This relationship between $v$ and $w$ implies that the nonzero eigenspaces of $\nabla T \circ \nabla T$ are even-dimensional. Indeed, assume there is another linearly independent eigenvector, $x$, with eigenvalue $\lambda$; we may assume that $x$ has unit length and is orthogonal to $v,w$. Then so is
$$
z \defeq \cd{x/{\sqrt{-\lambda}}}{T},
$$
because
$$
g(v,z) = g(v,\cd{x/{\sqrt{-\lambda}}}{T}) = -\frac{1}{\sqrt{-\lambda}}g(x,\cd{v}{T}) = -g(x,w) = 0;
$$
similarly, $g(w,z) = g(x,z) = 0$, so that the $\lambda$-eigenspace must be at least four-dimensional, and so on. (By self-adjointness, eigenvectors belonging to different eigenspaces will also be orthogonal.) With $v,w$, and $-\sqrt{-\lambda}$ playing the roles of $X_i,X_{i+1}$, and $f_{ii+1}$ in \eqref{eqn:even}, respectively, this proves the existence of a $\gL$-orthonormal basis $\{T,X_1,\dots,X_{n-1}\}$ satisfying \eqref{eqn:even}. (That there are $\leq \lfloor\frac{n-1}{2}\rfloor$ such functions $f_{ii+1}$ now follows.) For the remainder of the proof, assume for the moment that $h\defeq -\frac{1}{2}\gL(T,T)$ is not necessarily constant. Then the following equation is well known: For all vectors $V,W$,
\beqa
\label{eqn:known}
(\text{Hess}\,h)(V,W) = \gL(\cdL{V}{T},\cdL{W}{T}) - \RmL(V,T,T,W).
\eeqa 
(By analyzing the points on $M$ where $\text{Hess}\,h$ has an extremum, \eqref{eqn:known} leads to Berger's classical result that if an even-dimensional compact Riemannian manifold has positive sectional curvature, then every Killing vector field must have a zero; see, e.g., \cite[Theorem~8.3.1]{PP}. An analogous result, also relying on \eqref{eqn:known}, was recently shown to hold in the Lorentzian setting, in \cite[Theorem~3(A)]{RS2}.) Indeed, it is straightforward to verify that $\cdL{T}{T} = -\text{grad}h$, in which case
\beqa
(\text{Hess}\,h)(V,W) \!\!&=&\!\! \gL(\cdL{V}{\text{grad}h},W)\nonumber\\
&=&\!\! -\gL(\cdL{V}{\cdL{T}{T}},W)\nonumber\\
&=&\!\! -\RmL(V,T,T,W) - \hspace{-.35in}\underbrace{\,\gL(\cdL{T}{\!\cdL{V}{T}},W)\,}_{\text{$T(\gL(\cdL{V}{T},W)) - \gL(\cdL{V}{T},\cdL{T}{W})$}}\hspace{-.35in} - \overbrace{\,\gL(\cdL{[V,T]}{T},W)\,}^{\text{$-\gL(\cdL{W}{T},[V,T])$}}\nonumber\\
&=&\!\! -\RmL(V,T,T,W) + \gL(\cdL{V}{T},\cdL{W}{T})\nonumber\\
&&+\underbrace{\,\Big[\gL(\cdL{V}{T},\cdL{T}{W})-\gL(\cdL{W}{T},\cdL{T}{V})-T(\gL(\cdL{V}{T},W))\Big]\,}_{\text{$(*)$}}.\nonumber
\eeqa
Since $(*)$ is anti-symmetric in $V,W$ while the other terms are symmetric in $V,W$, we must have that $(*) = 0$, thus proving \eqref{eqn:known}. Reimposing now the unit-length condition $\gL(T,T) = -1$ yields $\text{Hess}\,h = 0$, from which it immediately follows from \eqref{eqn:even} and \eqref{eqn:known} that
$$
\RmL(X_i,T,T,X_j) = 0 \comma i\neq j
$$ 
and that
$$
\RmL(X_{i},T,T,X_{i}) = f_{ii+1}^2 = \RmL(X_{i+1},T,T,X_{i+1}).
$$
Finally, to prove that $T(f_{ii+1}) = 0$, we compute as follows:
\beqa
T(-2f_{ii+1}) \!\!&=&\!\! T(\gL(T,[X_i,X_{i+1}]))\nonumber\\
&=&\!\! \gL(T,\cdL{T}{\cdL{X_i}{X_{i+1}}}) - \gL(T,\cdL{T}{\cdL{X_{i+1}}{X_{i}}})\nonumber\\
&=&\!\! \cancel{\RmL(T,X_i,X_{i+1},T)} +\underbrace{\,\gL(\cdL{X_i}{\!\cdL{T}{X_{i+1}}},T)\,}_{\text{$-\gL(\cdL{T}{X_{i+1}},\cdL{X_i}{T})$}}+\underbrace{\,\gL(\cdL{[T,X_i]}{X_{i+1}},T)\,}_{\text{$\gL(\cdL{X_{i+1}}{T},[T,X_i])$}}\nonumber\\
&&\!\!-\cancel{\RmL(T,X_{i+1},X_{i},T)} -\underbrace{\,\gL(\cdL{X_{i+1}}{\!\cdL{T}{X_i}},T)\,}_{\text{$-\gL(\cdL{T}{X_{i}},\cdL{X_{i+1}}{T})$}}-\underbrace{\,\gL(\cdL{[T,X_{i+1}]}{X_{i}},T)\,}_{\text{$\gL(\cdL{X_{i}}{T},[T,X_{i+1}])$}}\nonumber\\
&=&\!\! -\gL(\cdL{T}{X_{i+1}},\cdL{X_i}{T}) +\gL(\cdL{T}{X_{i}},\cdL{X_{i+1}}{T})\nonumber\\
&&\hspace{.5in}+\,\gL(\cdL{X_{i+1}}{T},\cdL{T}{X_i}) - \cancel{\gL(\cdL{X_{i+1}}{T},\cdL{X_i}{T})}\nonumber\\
&&\hspace{1in}-\,\gL(\cdL{X_{i}}{T},\cdL{T}{X_{i+1}}) + \cancel{\gL(\cdL{X_{i}}{T},\cdL{X_{i+1}}{T})}\nonumber\\
&=&\!\! -2f_{ii+1}\gL(\cd{T}{X_{i+1}},X_i) -2f_{ii+1}\gL(\cd{T}{X_i},X_{i+1})\nonumber\\
&=&\!\! 0.\nonumber
\eeqa
This completes the proof.
\end{proof}

(Note that in dimension two, a unit-length timelike Killing vector field is necessarily parallel.) We would now like to generalize Proposition \ref{prop:basis} to the case when $T$ may not necessarily have unit length. Recalling \eqref{eqn:conf}, we know that $T$ will be a unit-length Killing vector field with respect to the conformal metric $\TgL \defeq \frac{\gL}{-\gL(T,T)}$. If we then apply Proposition \ref{prop:basis} to $\TgL$, we will obtain the following frame for $\gL$:

\begin{prop}
Let $(M,\gL)$ be a stationary spacetime with timelike Killing vector field $T$, and let $\TgL \defeq e^{2k}\gL = \frac{\gL}{-\gL(T,T)}$ be the conformal metric with respect to which $T$ is a unit-length Killing vector field. Then at every point of $M$ there exists a $\TgL$-orthonormal frame $\{T,X_1,\dots,X_{n-1}\}$ satisfying
$$
\cdL{X_i}{T} = f_{ii+1} X_{i+1} - X_i(k)T  \commas \cdL{X_{i+1}}{T} = -f_{ii+1} X_i - X_{i+1}(k)T,
$$
with the functions $f_{ii+1}$ defined as in \eqref{eqn:even}\emph{;} here $\cdL{}{}$ is the Levi-Civita connection of $\gL$. With respect to this frame, and setting $h\defeq -\frac{1}{2}\gL(T,T)$, the components \emph{$\RmL(X_i,T,T,X_j)$} and \emph{$\widetilde{\text{Rm}}_{\scalebox{0.4}{$L$}}(X_i,T,T,X_j)$} are related by
\emph{\beqa
\RmL(X_i,T,T,X_j) \!\!&=&\!\! -(\text{Hess}\,h)(X_i,X_j) +2h\,\widetilde{\text{Rm}}_{\scalebox{0.4}{$L$}}(X_i,T,T,X_j)\delta_{ij}\nonumber\\
&&\hspace{.5in} - (2h)^{-1}X_i(h)X_j(h).\label{eqn:even2}
\eeqa}
\end{prop}

\begin{proof}
This follows from the conformal formula between $\cdL{}{}$ and $\widetilde{\nabla}^{\scalebox{0.4}{$L$}}$; writing $\gL = e^{2(-k)}\TgL$ with $-k = \frac{1}{2}\text{ln}(-\gL(T,T))$, we have
$$
\cdL{X_i}{T} = \widetilde{\nabla}^{\scalebox{0.4}{$L$}}_{\!X_i}{T} + X_i(-k)T + \cancelto{0}{T(-k)}X_i - \cancelto{0}{\TgL(X_i,T)}\widetilde{\text{grad}}(-k),
$$
which yields $f_{ii+1}X_{i+1} -X_i(k)T$; similarly with $\cdL{X_{i+1}}{T}$. Inserting these into \eqref{eqn:known} and using Proposition \ref{prop:basis} then yields \eqref{eqn:even2}.
\end{proof}

More generally, the conformal formula between $\widetilde{\text{Rm}}_{\scalebox{0.4}{$L$}}$ and $\RmL$ yields
$$
\widetilde{\text{Rm}}_{\scalebox{0.4}{$L$}} = e^{2k}\Big[\RmL - \text{Hess}\,k\,{\tiny \owedge}\,\gL +(dk\otimes dk)\,{\tiny \owedge}\,\gL-\frac{1}{2}\langle dk,dk\rangle_{\scalebox{0.4}{$L$}}(\gL\,{\tiny \owedge}\,\gL)\Big],
$$
though we will forego analyzing this and instead return to the case when $\gL(T,T) = -1$. Thus, take an $(M,\gL)$ with $T$ unit-length and not parallel, and suppose that $M$ is three-dimensional; then by Proposition \ref{prop:basis}, at every point of $M$ there exists a local $\gL$-orthonormal frame $\{T,X_1,X_2\}$ satisfying
\beqa
\label{eqn:nice0}
\cdL{T}{T} = 0 \comma \cdL{X_1}{T} = fX_2 \comma \cdL{X_2}{T} = -fX_1,
\eeqa
for some $f < 0$. If $M$ is four-dimensional, then at every point there exists a local $\gL$-orthonormal frame $\{T,X_1,X_2,X_3\}$ satisfying
\beqa
\label{eqn:nice1}
\cdL{T}{T} = \cdL{X_1}{T} = 0 \comma \cdL{X_2}{T} = fX_3 \comma \cdL{X_3}{T} = -fX_2,
\eeqa
for some function $f < 0$.  If five- or six-dimensional, then there are at most two such functions, etc. Observe in particular that \eqref{eqn:nice1} generalizes to all even dimensions: If $n$ is even, then the vector fields $X_1,\dots,X_{n-1}$ are odd in number; because the functions $f_{ii+1}$ occur in pairs, at least one $X_i$ must satisfy $\cdL{X_i}{T} = 0$. In particular, at least one $\RmL(X_i,T,T,X_i) = 0$ in all even dimensions. In any case, the virtue of the orthonormal frames \eqref{eqn:even} is the way in which they simplify many of the curvature components between $\Rmr$ and $\RmL$ appearing in \eqref{eqn:RmL}. We now return to an analysis of this relationship, and thence to our main result.

\section{Main Result}
\label{sec:main}

\begin{lemma}
\label{lemma:frame}
The local orthonormal frame of Proposition \ref{prop:basis} has the following properties:
\begin{enumerate}[leftmargin=*]
\item[i.] \emph{$\Rmr(X_i,T,T,X_j) = -\RmL(X_i,T,T,X_j)\big(1+2\delta_{ij}\big) = 0$}.
\item[ii.] If exactly three of the indices $i,j,k,l$ are distinct, then
\emph{$$
\Rmr(X_i,X_j,X_k,X_l) = \RmL(X_i,X_j,X_k,X_l).
$$}
\end{enumerate}
\end{lemma}

\begin{proof}
We use Proposition \ref{prop:basis}. If $\cdL{X_i}{T} = fX_{i+1}$, then \mbox{$\gL(\cdL{X_i}{T},\cdL{X_j}{T}) = \delta_{ij}f$,} which proves i. If exactly three of $i,j,k,l$ are distinct, then each of the products $\gL(\cdL{X_i}{T},X_l)\gL(\cdL{X_j}{T},X_k)$, $\gL(\cdL{X_i}{T},X_k)\gL(\cdL{X_j}{T},X_l)$, and $\gL(\cdL{X_i}{T},X_j)\gL(\cdL{X_k}{T},X_l)$ in \eqref{eqn:RmL} vanishes; this proves ii. 
\end{proof}

Thanks to such simplifications, we have the following curvature and topological results in the Lorentzian setting:

\begin{thm}
\label{thm:Tmain4}
Let $(M,\gL)$ be a closed, connected, and oriented Lorentzian 3-manifold with a unit-length timelike Killing vector field $T$. Relative to the local $\gL$-orthonormal frame \eqref{eqn:nice0}, set \emph{$\RmL(T,X_i,X_j,X_k) \defeq R_{Tijk}$}, etc., where \emph{$\RmL$} is the Riemann curvature 4-tensor of $\gL$.  If the symmetric matrix
\beqa
\label{eqn:Rfinal0}
\begin{bmatrix}
-R_{T1T1} & 0 & R_{12T1}\\
0 & -R_{T2T2} & R_{12T2}\\
R_{T112} & R_{T212} & -R_{1212}+3R_{T1T1}+3R_{T2T2}
\end{bmatrix}
\eeqa
is 2-positive, then $b_1 = b_2 = 0$. In general, if $(M,\gL)$ is odd-dimensional and the ${n\choose 2}\times {n \choose 2}$ analogue of this matrix is $(n-p)$-positive for $1 \leq p \leq \lfloor \frac{n}{2}\rfloor$, then $b_1=\cdots=b_p=b_{n-p}=\cdots=b_{n-1}=0$. If $(M,\gL)$ is four-dimensional, then relative to the local $\gL$-orthonormal frame \eqref{eqn:nice1}, the symmetric matrix
\beqa
\label{eqn:Rfinal}
\begin{bmatrix}
-R_{T1T1} & R_{T2T1} & R_{T3T1} & R_{23T1} & R_{31T1} & R_{12T1}\\
R_{T1T2} & -R_{T2T2} & R_{T3T2} & R_{23T2} & R_{31T2} & R_{12T2}\\
R_{T1T3} & T_{T2T3} & -R_{T3T3} & R_{23T3} & R_{31T3} & R_{12T3}\\
R_{T123} & R_{T223} & R_{T323} & -R_{2323}+3R_{T2T2}+3R_{T3T3} & -R_{3123} & -R_{1223}\\
R_{T131} & R_{T231} & R_{T331} & -R_{2331} & -R_{3131} & -R_{1231}\\
R_{T112} & R_{T212} & R_{T312} & -R_{2312} & -R_{3112} & -R_{1212}
\end{bmatrix}
\eeqa
cannot be 3-positive. If $(M,\gL)$ is $2n$-dimensional, then the ${2n\choose 2}\times {2n \choose 2}$ analogue of this matrix cannot be $n$-positive\emph{;} if it is $(2n-p)$-positive for $1 \leq p < n$, then $b_1=\cdots=b_p=b_{2n-p}=\cdots=b_{2n-1}=0$.
\end{thm}

\begin{proof}
Let us start with the case when $(M,\gL)$ is three-dimensional: Together with \eqref{eqn:nice0}, Propositions \ref{prop:R}, \ref{prop:basis}, and Lemma \ref{lemma:frame}, the curvature operator $\co$ of the corresponding Riemannian metric \eqref{eqn:Riem} is easily verified to be \eqref{eqn:Rfinal0}, to which we now apply \cite{PW}. With $n=3$ and $1 \leq p \leq \lfloor \frac{3}{2}\rfloor$, the only option is $p=1$: If $\co$ is $(3-1)=2$-positive, then $b_1=b_{2} = 0$. The general odd-dimensional case, with a corresponding basis \eqref{eqn:even} guaranteed by Proposition \ref{prop:basis}, follows similarly, although the matrix will be more complicated than \eqref{eqn:Rfinal0}. Now suppose that $(M,\gL)$ is four-dimensional. Switching to the local frame given by \eqref{eqn:nice1}, $\co$ is easily verified to be \eqref{eqn:Rfinal}. With $n=4$ and $1 \leq p \leq \lfloor \frac{4}{2}\rfloor$, the options are now $p=1$ and $p=2$.  Consider $p=1$: If $\co$ is $(4-1)=3$-positive, then $b_1=b_{3} = 0$. To determine $b_2$, recall that a compact manifold $M$ admits a Lorentzian metric if and only its Euler characteristic is zero: $\chi(M) = 0 = b_0-b_1+b_2-b_3+b_4$ (see, e.g., \cite[Prop.~37,~p.~149]{o1983}).
However, by Poincar\'e duality and connectedness, $b_4 = b_0 =1$, hence this equation cannot be solved for (nonnegative) $b_2$. Thus $\co$ cannot be $3$-positive, and therefore cannot be $2$-positive, either. Indeed, if $\co$ is $(4-2)=2$-positive, then \cite{PW} dictates that $b_1 = b_2 = b_3 = 0$, which once again cannot occur. And so on in dimension $2n$: Regardless of the $f_{ii+1}$'s in the local frame guaranteed by Proposition \ref{prop:basis}, if $\co$ is $(2n-\lfloor \frac{2n}{2}\rfloor) =n$-positive, then $b_1=\cdots=b_n = \cdots = b_{2n-1} = 0$, by \cite{PW}; but then we would have $\chi(M) = 0 = b_0 +b_{2n} = 2$, a contradiction. (Note that, for $1 \leq p < n$, one has $\chi(M) = 0 = 1 - b_1 +b_2 - \cdots \pm b_n \cdots $, with $+$ in dimensions $2n = 4, 8, 12,\dots$ and $-$ in dimensions $2n=6,10,14,\dots$; for the former, only $1 \leq p \leq n-2$ is allowed, whereas for the latter, $1 \leq p \leq n-1$ is allowed.) 
\end{proof}

As $n$ increases, the following pattern occurs: If $(M,\gL)$ is five-dimensional and its corresponding $10 \times 10$ matrix (with at most two functions from \eqref{eqn:even}) is \emph{4-positive}, then $b_1 = b_4 = 0$; if it is \emph{3-positive}, then $b_1=b_2=b_3=b_4=0$. If $(M,\gL)$ is six-dimensional and its now $15 \times 15$ matrix (once again with at most two functions from \eqref{eqn:even}) is \emph{5-positive}, then $b_1 = b_5 = 0$; if it is \emph{4-positive}, then $b_1=b_2=b_4=b_5=0$ and $b_3 = 2$ (by Theorem \ref{thm:Tmain4}, it cannot be \emph{3-positive}). And so on in higher dimensions. Regardless of the positivity, for each $n$ these matrices should be compared with the corresponding\,---\,generally non-symmetric\,---\,Lorentzian curvature operator $\hat{R}_{\scalebox{0.4}{$L$}}$. The closest alignment between the two occurs in dimensions three (recall \eqref{eqn:sym0}) and four:
\beqa
\label{eqn:sym1}
\hat{R}_{\scalebox{0.4}{$L$}} = \begin{bmatrix}
R_{T1T1} & R_{T2T1} & R_{T3T1} & R_{23T1} & R_{31T1} & R_{12T1}\\
R_{T1T2} & R_{T2T2} & R_{T3T2} & R_{23T2} & R_{31T2} & R_{12T2}\\
R_{T1T3} & R_{T2T3} & R_{T3T3} & R_{23T3} & R_{31T3} & R_{12T3}\\
-R_{T123} & -R_{T223} & -R_{T323} & -R_{2323} & -R_{3123} & -R_{1223}\\
-R_{T131} & -R_{T231} & -R_{T331} & -R_{2331} & -R_{3131} & -R_{1231}\\
-R_{T112} & -R_{T212} & -R_{T312} & -R_{2312} & -R_{3112} & -R_{1212}
\end{bmatrix}\cdot
\eeqa
We conclude with the following application of our Theorem:

\begin{cor}
Let $(M,\gL)$ be a closed, connected, and oriented Lorentzian 3-manifold with a unit-length timelike Killing vector field $T$, Ricci tensor \emph{$\RicL$}, and scalar curvature \emph{$\text{scal}_{\scalebox{0.4}{$L$}}$}. If $T$ is not parallel, and if
\emph{
\beqa
\label{eqn:Ric0}
\text{scal}_{\scalebox{0.4}{$L$}}> 3\RicL(T,T) + |d\,\text{ln}(\RicL(T,T))|_{g_{\scalebox{0.3}{$L$}}}^2,
\eeqa}then $b_1=b_2=0$. If $\gL$ is replaced by a Riemannian metric $g$, then the inequality becomes \emph{$\text{scal} > \Ric(T,T) + |d\,\text{ln}(\Ric(T,T))|_{g}^2$}.
\end{cor}

\begin{proof}
By Proposition \ref{prop:basis}, there is a local orthonormal frame $\{T,X,Y\}$ satisfying 
$\cdL{X}{T} = fY, \cdL{Y}{T} = -fX$ for some smooth function $f < 0$ (nonzero because $T$ is not parallel) which satisfies 
$$
\RicL(T,T) = \RmL(X,T,T,X) + \RmL(Y,T,T,Y) = 2f^2.
$$
Recall also that $T(f) = 0$. Inserting this data into \eqref{eqn:Rfinal0} yields
\beqa
\begin{bmatrix}
-R_{T1T1} & R_{T2T1} & R_{12T1}\\
R_{T1T2} & -R_{T2T2} & R_{12T2}\\
R_{T112} & R_{T212} & -R_{1212}+3R_{T1T1}+3R_{T2T2}
\end{bmatrix} \!\!&=&\!\! \nonumber\\
&&\hspace{-2in}\begin{bmatrix}
f^2 & 0 & -X(f)\\
0 & f^2 & -Y(f)\\
-X(f) & -Y(f) & -4f^2+\frac{1}{2}\text{scal}_{\scalebox{0.4}{$L$}}
\end{bmatrix}\cdot\label{eqn:check3}
\eeqa
The eigenvalues of \eqref{eqn:check3} are
$$
f^2 \comma \underbrace{\,\frac{1}{2}\bigg(\!\!-\!3f^2+\frac{\text{scal}_{\scalebox{0.4}{$L$}}}{2} \pm \sqrt{\Big(5f^2-\frac{\text{scal}_{\scalebox{0.4}{$L$}}}{2}\Big)^{\!2}+4\big(X(f)^2+Y(f)^2\big)}\,\bigg)\,}_{\text{$\lambda_{\pm}$}}\cdot
$$
With $1 \leq p \leq \lfloor \frac{3}{2}\rfloor$, the only option is $p=1$; therefore, we determine when \eqref{eqn:check3} will be $(3-1)=2$-positive. Clearly $\lambda_- \leq \lambda_+$; in fact $f^2 \leq \lambda_+$ also, since
$$
\lambda_+ \geq f^2 \iff \sqrt{\Big(5f^2-\frac{\text{scal}_{\scalebox{0.4}{$L$}}}{2}\Big)^{\!2}+4\big(X(f)^2+Y(f)^2\big)} \geq 5f^2 - \frac{\text{scal}_{\scalebox{0.4}{$L$}}}{2}\cdot
$$
Thus the sum of the two smallest eigenvalues is
$$
\lambda_-+f^2 = -2f^2+ \frac{1}{2}\text{scal}_{\scalebox{0.4}{$L$}}-\lambda_+,
$$
where we've used the fact that the sum of all three eigenvalues is equal to the trace, $-2f^2+\frac{1}{2}\text{scal}_{\scalebox{0.4}{$L$}}$. The condition $-2f^2+\frac{1}{2}\text{scal}_{\scalebox{0.4}{$L$}}-\lambda_+>0$ is equivalent to
$$
-f^2+\frac{\text{scal}_{\scalebox{0.4}{$L$}}}{2} > \sqrt{\Big(5f^2-\frac{\text{scal}_{\scalebox{0.4}{$L$}}}{2}\Big)^{\!2}+4\big(X(f)^2+Y(f)^2\big)},
$$
which in turn is equivalent to
\beqa
\label{eqn:S}
\text{scal}_{\scalebox{0.4}{$L$}} > 6f^2 + \frac{X(f)^2+Y(f)^2}{f^2}\cdot
\eeqa
Since $2f^2 = \RicL(T,T)$ and
$$
d\,\text{ln}(\RicL(T,T)) = X\big(\text{ln}(\RicL(T,T))\big) X^{\flat} + Y\big(\text{ln}(\RicL(T,T))\big)Y^{\flat}, 
$$
the inequality \eqref{eqn:S} is precisely \eqref{eqn:Ric0}. By Theorem \ref{thm:Tmain4}, we must therefore have $b_1=b_2 = 0$. Finally, if $\gL$ is replaced by a Riemannian metric $g$, then \eqref{eqn:check3} becomes, after using \eqref{eqn:LCs} on the frame $\{T,X,Y\}$, the Riemannian curvature operator $\co$ takes the form
$$
\begin{bmatrix}
-R_{T1T1} & -R_{T2T1} & -R_{12T1}\\
-R_{T1T2} & -R_{T2T2} & -R_{12T2}\\
-R_{T112} & -R_{T212} & -R_{1212}
\end{bmatrix} = \begin{bmatrix}
f^2 & 0 & X(f)\\
0 & f^2 & Y(f)\\
X(f) & Y(f) & -2f^2+\frac{1}{2}\text{scal}
\end{bmatrix}\cdot
$$
As a consequence, \eqref{eqn:S} is modified to 
$$
\text{scal} > 2f^2 + \frac{X(f)^2+Y(f)^2}{f^2},
$$
from which the Riemannian inequality follows.
\end{proof}

As a check, one can verify that for the Lorentzian metric \eqref{cor:check} on $\mathbb{S}^3$, $\RicL(T,T)=2$ and $\text{scal}_{\scalebox{0.4}{$L$}} = 10$, so that \eqref{eqn:Ric0} is satisfied; similarly the standard Riemannian metric \eqref{eqn:round} on $\mathbb{S}^3$ satisfies $\Ric(T,T)=2$ and $\text{scal} = 6$, so that the Riemannian inequality is satisfied also. Finally, observe that if $T$ was parallel, then the universal cover of $(M,\gL)$ would split isometrically as $\RR\times N$, by the Lorentzian version of the de Rham splitting theorem \cite{wu}; similarly for the Riemannian case.

\section*{References}
\renewcommand*{\bibfont}{\footnotesize}
\printbibliography[heading=none]
\end{document}